\documentclass[11pt, a4paper]{article}
\usepackage{tikz}
\usepackage{amssymb,amsfonts,amsthm,amsmath,amssymb}
\usepackage{caption}
\usepackage{float}
\usepackage{epic}
\usepackage{graphics}
\usepackage{graphicx}
\usepackage[T1]{fontenc}
\usepackage{cite}

\usepackage[colorlinks,
linkcolor=blue,
anchorcolor=blue,
citecolor=blue
]{hyperref}

\usepackage[symbol]{footmisc}

\usepackage{enumerate}
\usepackage{indentfirst}
\usepackage[margin=2.9cm]{geometry}
\allowdisplaybreaks

\newtheorem{thm}{Theorem}[section]
\newtheorem{lem}[thm]{Lemma}
\newtheorem{prop}[thm]{Proposition}
\newtheorem{cor}[thm]{Corollary}

\usepackage{caption}

\newcommand{\pattern}{  \begin{minipage}[c]{1.45em}
        {\scalebox{0.5}{\includegraphics{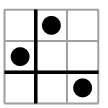}}}
    \end{minipage}   }

\begin{document}
\captionsetup[figure]{labelfont={bf},labelformat={default},labelsep={period},name={Fig.}}
	\begin{center}
		{\large \bf On enumeration of pattern-avoiding Fishburn permutations }
	\end{center}

	\begin{center}
		Yujie Du
		  and  Philip B. Zhang\footnote[1]{
            Email: \texttt{zhang@tjnu.edu.cn}

            This work was supported by the  National Natural Science Foundation of China (No.~12171362).
        }

	\end{center}

	\begin{center}
		College of Mathematical Science \\
		Tianjin Normal University, Tianjin  300387, P. R. China\\[6pt]
		\par\end{center}
	\noindent
	\textbf{Abstract.}
    In this paper, we prove two conjectures of Egge on the enumeration of several classes of pattern-avoiding Fishburn permutations. Our results include enumerating Fishburn permutations avoiding pattern 321 and one of the following three types of classical patterns: a pattern of size 4, two patterns of size 4, or a pattern of size 5.


	\noindent
	\textbf{AMS Classification 2020:} 05A15

	\section{Introduction}
	The main purpose of this paper is to prove certain enumeration formulae, conjectured by Egge~\cite{Egge2022PatternAvoidingFP}, about pattern-avoiding Fishburn permutations. A permutation of length $n$ is a rearrangement of the set $[n]:=\{1,2,\ldots,n\}$. We let $S_n$ denote the set of permutations of $[n]$. A permutation $\pi_1\pi_2\cdots\pi_n\in S_n$ avoids a {\em pattern} $p=p_1p_2\cdots p_k\in S_k$  if there is no subsequence $\pi_{i_1}\pi_{i_2}\cdots\pi_{i_k}$ such that $\pi_{i_j}<\pi_{i_m}$ if and only if $p_j<p_m$. Two permutations $\sigma$ and $\tau$ are {\em Wilf-equivalent} if for all $n$, the number of permutations of $S_n$ which avoid $\sigma$ is the same as the number of permutations of $S_n$ which avoid $\tau$.
    We refer the reader to the book by Kitaev~\cite{Kitaev2011} for a comprehensive introduction to the theory of permutation patterns.

The {\em Fishburn pattern}  $\pattern$
is a special case of {\em bivincular patterns} introduced in \cite{BousquetMlou200822freePA}. A permutation $\pi=\pi_1\pi_2\cdots\pi_n$ avoids $\pattern$ if there are no indices $i<j$ such that $\pi_{i}\pi_{i+1}\pi_{j}$ form a copy of the pattern 231 and $\pi_i=\pi_j+1$. {\em Fishburn permutations} are  $\pattern$-avoiding permutations.

We adopt the notation $F_n(\sigma_1,\ldots,\sigma_k)$ to represent the set of Fishburn permutations of length $ n $ that simultaneously avoid the patterns $\sigma_1,\ldots,\sigma_k$. For instance, $F_{n}(321,1243)$ represents the set of Fishburn permutations of length $n$ that avoid both the  patterns 321 and 1243.

Bousquet-M\'{e}lou, Claesson, Dukes, and Kitaev~\cite{BousquetMlou200822freePA} gave bijections between Fishburn permutations and {\em ascent sequences}, {\em unlabeled $(2+2)$-free posets}, and {\em linearized chord diagrams}. They also proved that the generating function for Fishburn permutations is \\[-4mm]
\[
1+\sum_{n\geq1}\prod_{i=1}^{n}(1-(1-t)^i).
\]
For the literature related to the Fishburn numbers, see~\cite{Stoimenow1998,Levande2013,Andrews2016,Fu2020}.

Recently, Cerbai and Claesson~\cite{Cerbai2020} developed the theory of the transport of patterns from Fishburn permutations to (modified) ascent sequences. The pattern avoidance on (modified) ascent sequences has been discussed recently by Duncan and Steingr\'{i}msson~\cite{Duncan2011}, and Cerbai~\cite{Cerbai2023,Cerbai2024}. 

In this paper, we are concerned  with enumeration of Fishburn permutations which  avoid specific  classical patterns.
Gil and Weiner~\cite{Gil2021} studied avoidance of classical patterns of size 3 or 4 on all and indecomposable Fishburn permutations. Egge~\cite{Egge2022PatternAvoidingFP} settled one of the conjectures in \cite{Gil2021} and enumerated Fishburn permutations that avoid specific classical pattern sets of size 3 or 4. Moreover, Egge~\cite{Egge2022PatternAvoidingFP}  proposed several conjectures on pattern-avoiding Fishburn permutations.

We prove Conjectures~10.14 and~10.17 in \cite{Egge2022PatternAvoidingFP} that involve a total of 15 enumerative results.
Table~\ref{tab-1} is a summary of the main results in this paper. Notably,
the formulas presented in the last four lines in Table~\ref{tab-1} have not been found or conjectured before, and we use  them to prove Egge's conjectures.

\begin{table}[]
        \renewcommand{\arraystretch}{1.8}
 \begin{center}
	\begin{tabular}{|c|c|c|}
		\hline
		Patterns $\sigma_1,\ldots,\sigma_k$ & |$F_{n}(\sigma_1,\ldots,\sigma_k)$|&Reference\\
		\hline \hline
		321, 1243 & $n^2-3n+4\quad(n\geq2)$&Theorem \ref{1243}\\[5pt]
		\hline
		321, 2134 & $n^2-3n+4\quad(n\geq2)$& Theorem \ref{2134} \\[5pt]
		\hline
		321, 1324 & $\frac{3}{2}n^{2}-\frac{13}{2}n+10\quad(n\geq3)$& Theorem \ref{1324}\\
		\hline  \hline
		321, 1423, 2143 & $\binom{n}{2}+1\quad(n\geq 0) $& Theorem \ref{1423-2143}\\
		\hline
		321, 3142, 2143& $ \binom{n}{2}+1\quad(n\geq 0) $& Theorem \ref{3142-2143}\\
		\hline
	321, 2143, 3124 &$\binom{n}{2}+1\quad(n\geq 0)$& Theorem \ref{2143-3124}\\
		\hline
	321, 2143, 4123 &$ \binom{n}{2}+1\quad(n\geq 0) $& Theorem \ref{2143-4123}\\
	\hline
	321, 1423, 3124&$F_n+2\quad(n\geq4)$& Theorem \ref{1423-3124}\\
	\hline
	321,  1423, 4123& $F_{n+1}-1\quad(n\geq1)$& Theorem \ref{1423-4123}\\
	\hline
	321, 3124, 4123& $F_{n+1}-1\quad(n\geq1)$& Theorem \ref{3124-4123}\\
	\hline  \hline
	321, 14253& $2^{n}-\binom{n}{2}-1\quad(n\geq1)$& Theorem \ref{321-14253}\\
	\hline
	321, 21354& $2^{n}-\binom{n}{2}-1\quad(n\geq1)$& Theorem \ref{321-21354}\\
	\hline
	321, 31452& $\frac{P_n+P_{n-1}+1}{2}\quad (n\geq1)$& Theorem \ref{321-31452}\\
	\hline
	321, 31524& $\frac{P_n+P_{n-1}+1}{2}\quad (n\geq1)$& Theorem \ref{31524}\\
	\hline
	321, 41523& $\frac{P_n+P_{n-1}+1}{2}\quad (n\geq1)$& Theorem \ref{321-41523}\\
	\hline \hline
	321, 132& $n\quad(n\geq1)$& Lemma~\ref{lem:321-132}\\
	\hline
    321, 213 & $n\quad(n\geq1)$& Lemma~\ref{lem:321-213}\\ \hline
    321, 312 & $F_{n}\quad(n\geq1)$& Lemma~\ref{lem:321-312}\\ \hline
    321, 3142& $2^{n-1}\quad (n\geq1)$& Lemma \ref{lem3142}\\
    \hline
	\end{tabular}
 	\caption{Enumeration of $\sigma$-avoiding Fishburn permutations}
\label{tab-1}
\end{center}
\end{table}

The following lemma shows that all the enumeration results discussed in this paper can be categorized into two types depending on the position of  the smallest element 1.

\begin{lem}\label{lem-12}
    Let $\pi\in F_n(321)$. Then either $\pi_1=1$ or $\pi_2=1$.
\end{lem}
\begin{proof}
    We shall prove this lemma by contradiction. Suppose that  the element 1 does not occupy one of the first two positions.  Let $\pi_1=x$ and $\pi_2=y$. If $x>y$ then $x,y$, and $1$ form a copy of the pattern 321. If $x<y$ then $x-1$ must appear to the right of $y$. This implies that the elements $x,y$, and $x-1$ form a copy of the Fishburn pattern  $\pattern$. Both cases lead to contradictions. Thus, we can conclude that $\pi_1=1$ or $\pi_2=1$.
\end{proof}

Since all cases of interest to us involve 321-avoiding Fishburn permutations, we classify  the proofs of the enumeration results into two types based on Lemma \ref{lem-12} and deal with them separately.
Throughout the paper we shall denote the set $\{\pi\in F_n(321, \sigma_1,\ldots,\sigma_k):\pi_i=1\}$ by $F_n^{(i)}(321, \sigma_1,\ldots,\sigma_k)$ for $i\in\{1,2\}$.

The following lemma shows that some pattern-avoiding Fishburn permutations can be deduced from classical results. We shall use it to prove Lemmas~\ref{lem:321-132}, \ref{lem:321-213}, \ref{lem:321-312}, \ref{lem3142}, and Theorem~\ref{3142-2143}.
\begin{lem}\label{lem:classical_pattern}
	For each pattern $\sigma\in\{132,213,312,3142\}$, we have $$F_n(321,\sigma)=S_n(231,321,\sigma).$$
\end{lem}

\begin{proof}

The inclusion $S_n(231,321,\sigma)\subseteq F_n(321,\sigma)$ is trivial since the $\pattern$-pattern is contained in the 231-pattern. To prove the opposite inclusion, it suffices to show that if a permutation $\pi\in F_n(321,\sigma)$ avoids $\pattern$ then $\pi$ avoids 231. We consider its contrapositive form. Let $a,b,c$ be a copy of 231, as illustrated in the picture below, satisfying that $a$  and $c$ are the smallest and largest elements, respectively.
\begin{figure}[htpb]
	\centering
	\begin{tikzpicture}[scale=0.8]
		\draw (0.01,0.01) grid (3+0.99,3+0.99);
		\fill (1,2) circle (2pt);
		\fill (2,3) circle (2pt);
		\fill (3,1) circle (2pt);
		\node[below right=2pt] at (0.4,2.6) {$a$};
		\node[below left=2pt] at (2.6,3.6) {$b$};
		\node[above left=2pt] at (3.6,0.4) {$c$};
		\node at (1.5,3.5) {$A$};
		\node at (1.5,2.5) {$B$};
		\node at (1.5,1.5) {$C$};
		\node at (1.5,0.5) {$D$};
		\node at (0.5,1.5) {$E$};
		\node at (2.5,1.5) {$F$};
		\node at (3.5,1.5) {$G$};
	\end{tikzpicture}
\end{figure}

	To avoid 321, there are no elements in regions $A,C,F$. Moreover, there are no elements in regions $E$ and $G$ due to the selection of $a$ and $c$. Thus, we have $a=c+1$. Next, we prove that there are no elements in region $D$. If an element $x$ is in region $D$, then the elements $xbc, axb, axc, axbc$ would form a copy of pattern 132, 213, 312 and 3142, respectively. So we can find a copy of $\pattern$ consisting of $a$, the element immediately to its right (namely the first element in region $B$ or $b$ when $B$ is empty) and $c$.
\end{proof}

Recall that the \emph{direct sum} of two permutations $\sigma$ and $\tau$, of lengths $k$ and $\ell$ respectively, denoted by $\sigma\oplus\tau$, is a permutation of length  $k+\ell$ consisting of $\sigma$ followed by $\tau^{\prime}$, where $\tau'$ is obtained from $\tau$ by adding $k$ to each element. For example, $12\oplus123=12345$. Additionally, an \emph{active site} in a permutation $\pi\in F_{n-1}(\sigma)$ is the space to the left or to the right of $\pi$ or between two consecutive elements in $\pi$ such that the permutation obtained by inserting $n$ into this space belongs to $F_{n}(\sigma)$.

The following observation is very useful for us.
Suppose $\sigma$ is a pattern (or a set of patterns) which does not begin with 1.
One can easily see that $|F_n^{(1)}(321, \sigma)| = |F_{n-1}(321,\sigma)|$.
If $\sigma$ begins with 1 then we can make use of the fact (which can also be proved directly) that $|F_n^{(1)}(321, 1 \oplus \tau)| = |F_{n-1}(321,\tau)|$.

This paper is organized as follows. In Section~\ref{sec2}, we prove three enumerative results, Theorems \ref{1243},  \ref{2134}, and \ref{1324}, on Fishburn permutations avoiding the pattern 321 and  a classical pattern of size 4.
In Section~\ref{sec3}, we consider the enumeration of Fishburn permutations that avoid the pattern 321 and two classical patterns of size 4.
Theorems \ref{1423-3124},  \ref{1423-4123}, and \ref{3124-4123} involve the Fibonacci numbers.
In Section~\ref{sec4}, we deal with pattern avoidance of size 5 and our results in Theorems~\ref{321-31452}, \ref{31524}, and  \ref{321-41523} are related to the Pell numbers.

\section{Avoiding 321 and a classical pattern of size 4}
\label{sec2}

\subsection{Enumerating $ F_{n}(321,1243)$}

The main result in this subsection is the following theorem.
\begin{thm}\label{1243}
	For $n\geq 2$,  $|F_{n}(321,1243)|=n^{2}-3n+4$.
\end{thm}


Before proving Theorem \ref{1243}, we shall first enumerate $F_n(321,132)$.
\begin{lem}\label{lem:321-132}
    For $n\geq 1$, we have
    \[
    |F_n(321,132)|=n.
    \]
\end{lem}

\begin{proof}
	By Lemma~\ref{lem:classical_pattern}, the Fishburn permutations avoiding patterns 321 and 132 are in fact permutations avoiding patterns 231, 321 and 132. Simion and Schmidt have proved in \cite[Lemma 6 (b)]{Simion1985} that $|S_n(231,321,132)|=n$. Thus we finish the proof.
\end{proof}

\begin{prop}\label{1243-2}
	For $n\geq 2$, $|F_{n}^{(2)}(321,1243)|=n^{2}-4n+5$.
\end{prop}

\begin{proof}
	For $n=2$, $F_2^{(2)}(321,1243)=\{21\}$. Thus $|F_{2}^{(2)}(321,1243)|=1=2^{2}-4\times 2+5$.

	For $n=3$, $F_3^{(2)}(321,1243)=\{213,312\}$. Thus $|F_{3}^{(2)}(321,1243)|=2=3^{2}-4\times 3+5$.

	For $n=4$, $F_4^{(2)}(321,1243)=\{2134,2143,3124,3142,4123\}$. Thus $$|F_{4}^{(2)}(321,1243)|=5=4^{2}-4\times 4+5.$$

	For $n\geq 5$, let $\pi_1=k$. We  consider the following cases of possible values of $k$.
\begin{enumerate}[1)]
\item$k=n$. To avoid 321,  the elements $2,\ldots,n-1$ appearing after $n$ must be in increasing order. So $\pi=n\ 1\ 2\cdots n-1$.


\item $3\leq k\leq n-1$. Suppose $\pi=k\ 1\ \pi_{3}\pi_{4}\cdots \pi_{n}$. We consider the value of~$\pi_3$.
\begin{enumerate}[a)]
\item $\pi_3<k$. In this case, $\pi_{3}=2$, because otherwise $k,\pi_{3},2$ would be a copy of 321. Additionally, to avoid 1243, the elements $3,\ldots,k-1,k+1,\ldots,n$ to the right of 1 and 2 must be in increasing order. Therefore, $$\pi=k\ 1\ 2\ 3 \cdots k-1\ k+1\cdots n.$$

\item $\pi_{3}=k+1$. To avoid 321, the elements $2,\ldots,k-1$ are in increasing order. Meanwhile, the elements $k+2,\ldots,n$ are in increasing order. If not, $1,k+1$ and a descending pair would be a copy of 1243. In addition, no elements of $\{k+2,\ldots,n\}$ can be inserted into the site between $2$ and $k-1$. Indeed, if some $x\in \{k+2,\ldots,n\}$  is inserted into the sites between $2$ and $k-1$, say $x$ is inserted into the site between $i$ and $i+1$  ($2\leq i\leq k-2$), then $1,i,x,i+1$ form a copy of 1243, which is impossible.  It remains to consider $2\cdots k-1$ as a whole part and insert it into
$$k\ 1\ k+1\ ^{1}\ k+2\ ^{2}\ k+3\cdots \, ^{n-k-1} n\, ^{n-k}.$$ All the  $n-k$ sites labeled above are  available to insert $2\cdots k-1$.
So in this case, the number of permutations is $n-k$.

\item $\pi_{3}\geq k+2$. We claim that $\pi_{4}=2$. If $2<\pi_{4}<\pi_{3}$ then $\pi_{3},\pi_{4},2$ form a copy of 321. Meanwhile, if $\pi_{4}>\pi_{3}$ then $\pi_{3},\pi_{4},\pi_{3}-1$ form a copy of  $\pattern$. Similar to the argument in the previous  paragraph, $2,\ldots,k-1$ are in increasing order and, to avoid 1243, no elements in $\{k+1,\ldots,\pi_{3}-1,\pi_{3}+1,\ldots,n\}$ can be inserted into the sites between $2$ and $k-1$. Thus  $k+1,\ldots,\pi_{3}-1,\pi_{3}+1,\ldots,n$ are in increasing order to the right of $k-1$. So $$\pi=k\ 1\ \pi_{3}\ 2\cdots k-1\ k+1\cdots \pi_{3}-1\ \pi_{3}+1\cdots n.$$ As $\pi_3$ ranges from $k+2$ to $n$, the number of permutations in  this subcase is $n-k-1$.
\end{enumerate}
For a fixed $k$, the total number of permutations given by the above three forms is $$1+(n-k)+(n-k-1)=2(n-k).$$
\item $k=2$.
We can consider 2 and 1 as a single element playing the role of 1.
Thus, the number of permutations is $|F_{n-2}(321,132)|=n-2$.
\end{enumerate}
	Therefore, for $n\geq 5$,
	\[|F_{n}^{(2)}(321,1243)|=1+\sum_{k=3}^{n-1} 2(n-k)+(n-2) =n^{2}-4n+5.\]
The proposition is thus proved.
\end{proof}

\begin{proof}[Proof of Theorem \ref{1243}]
For $n\geq 2$, we obtain
\[|F_{n}^{(1)}(321,1243)| = |F_{n-1}(321,132)|,\]
and thus, by Lemma~\ref{lem:321-132} and Proposition~\ref{1243-2},
	\[
	\begin{aligned}
		|F_{n}(321,1243)|=&|F_{n}^{(1)}(321,1243)|+|F_{n}^{(2)}(321,1243)|
		\\=&(n-1)+(n^2-4n+5)\\
		=&n^{2}-3n+4.
	\end{aligned}
	\]

\vspace{-1cm}

\ \ \
\end{proof}

\subsection{Enumerating $F_{n}(321,2134)$}
The main result in this subsection is as follows.
\begin{thm}\label{2134}
	For $n\geq 2$,  $|F_{n}(321,2134)|=n^{2}-3n+4$.
\end{thm}

	\begin{prop}
	For $n\geq 3$, $|F_{n}^{(2)}(321,2134)|=2n-4$.
\end{prop}
\begin{proof}
	For $n=3$, $|F_{3}^{(2)}(321,2134)|=2=2\times 3-4$.

For $n\geq 4$, we let $\pi_1=k$. We first claim $k\geq n-2$ when $\pi_2=1$. Otherwise, $k\leq n-3$. To avoid 2134, the elements $k+1,\ldots,n$ appearing after $k\ 1$ must be in descending order. However, since $|\{k+1,\ldots,n\}|=n-k\geq 3$, there exists a forbidden copy of 321, which $\pi$ avoids. Thus $k\geq n-2$ as claimed.   We proceed to discuss according to the three values of $k$ for $n\geq 4$.
	\begin{enumerate}[1)]
		\item $k=n$. The elements $2,\ldots,n-1$ that appear after $n$ must be increasing to avoid 321. So the form of $\pi$ is  $n\ 1\ 2\cdots n-1$.

		\item 	$k=n-1$. Similarly, the elements $2,\ldots,n-2$ that are positioned to the right of $n-1$  are arranged in increasing order. Suppose that $\pi^{\prime}=n-1\ 1\ ^{1}\ 2\ ^{2}\cdots\ ^{n-3}\ n-2\ ^{n-2}$. We can insert $n$ into each of the  $n-2$ active sites to obtain $\pi$. Thus, the number of permutations in this case is $n-2$.
		\item $k=n-2$. To avoid 321, the elements $2,\ldots,n-3$ are in increasing order. Let $$\pi^{\prime}=n-2\ 1\ ^{1}\ 2\ ^{2}\cdots\ ^{n-4}\ n-3\ ^{n-3}.$$ To construct $\pi$ we insert $n-1$ and $n$ into the sites labeled  from 1 to $n-3$. First, $n$ must be placed to the left of $n-1$ to avoid the pattern 2134.  Second, $n-1$ cannot appear in the $i$-th site for $1\leq i\leq n-4$ since $n,n-1,n-3$ form a copy of 321. Thus $n-1$ can only be inserted to the $(n-3)$-th site. Let $$\pi^{\prime\prime}=n-2\ 1\ ^{1}\ 2\ ^{2}\cdots\ ^{n-4}\ n-3\ ^{n-3}\ n-1.$$  We should insert $n$ into $\pi^{\prime\prime}$ to obtain $\pi$. Notice that $n$ can  be inserted into each of the $n-3$ sites. Thus, the number of permutations in this case is $n-3$.
	\end{enumerate}
	Therefore, for $n\geq 4$, we have
	\[|F_{n}^{(2)}(321,2134)|=1+(n-2)+(n-3)=2n-4.\]

    \vspace{-1cm}

    \ \ \
\end{proof}

\begin{proof}[Proof of Theorem \ref{2134}]
	For $n=2$, $F_{n}(321,2134)=\{12,21\}$, and thus $$|F_{2}(321,2134)|=2=2^{2}-3\times 2+4.$$
  For $n\geq 3$, by induction, we have that  $$|F_{n}^{(1)}(321,2134)|=|F_{n-1}(321,2134)|=n^{2}-5n+8.$$
	Thus, we get
	\[
	\begin{aligned}
		|F_{n}(321,2134)|=&|F_{n}^{(1)}(321,2134)|+|F_{n}^{(2)}(321,2134)|
		\\=&(n^2-5n+8)+(2n-4)\\
		=&n^{2}-3n+4,
	\end{aligned}
	\]which completes the proof.
\end{proof}

\subsection{Enumerating $F_{n}(321,1324)$}
The main result of this subsection is as follows.
\begin{thm}\label{1324}
	For $n\geq 3$,  $|F_{n}(321,1324)|=\dfrac{3}{2}n^{2}-\dfrac{13}{2}n+10$.
\end{thm}

\begin{lem}\label{lem:321-213}
		For $n\geq 1$, $|F_{n}(321,213)|=n$.
	\end{lem}

\begin{proof}
	By Lemma~\ref{lem:classical_pattern}, we have $|F_n(321,213)|=|S_n(231,321,213)|$. It is easy to see that the pairs $ (231,321,213) $ and $ (213,123,231) $ are Wilf-equivalent by applying the complement operation. In~\cite[Lemma 6 (d)]{Simion1985}, Simion and Schmidt proved that $|S_n(213,123,231)|=n$. Thus, we have $|F_{n}(321,213)|=n$.
\end{proof}

	\begin{prop}\label{prop:321-1324-2}
		For $n\geq 3$, $|F_{n}^{(2)}(321,1324)|=\dfrac{3}{2}n^{2}-\dfrac{15}{2}n+11$.
	\end{prop}

	\begin{proof}
		We can check this claim directly for $n\leq 5$.
	Let $\pi_1=k$. We consider possible values of $k$ for $n\geq 6$.
		\begin{enumerate}[1)]
			\item $k=2$. We can consider the first two elements 2 and 1 as a single element playing the role of 1. Thus by Lemma~\ref{lem:321-213}, the number of permutations in this case is $|F_{n-2}(321,213)|=n-2$.

\item $3\leq k\leq n-2$. Avoiding 321 implies that $2,\ldots,k-1$ to the right of $k\ 1$ must be in increasing order. Next we consider the placement of elements $k+1,\ldots,n$.
\begin{enumerate}[a)]
\item None of the elements in $\{k+1,\ldots,n\}$ is to the left of $k-1$. Then we can consider $k\ 1\ 2\cdots k-1$ as the smallest element 1. By Lemma \ref{lem:321-213},
the number of permutations in this subcase is $|F_{n-k}(321,213)|=n-k$.

\item Exact one element, denoted by $x$, in $\{k+1,\ldots,n\}$ is to the left of $k-1$. We claim  that $x=n$. Otherwise, $k+1\leq x<n$. Thus $1,x,k-1,n$ form a copy of 1324, which $\pi$ avoids. Therefore, $k+1,\ldots,n-1$ appearing after $n$ must be in increasing order. This implies that $n$ can be inserted into each site between 1 and $k-1$ as follows
\[
\pi^\prime=k\ 1\ ^{1}\ 2\ ^{2}\cdots\ ^{k-2}\ k-1\ k+1\cdots n-1.
\]
Thus the number of permutations in this subcase is $k-2$.

\item At least two elements in $\{k+1,\ldots,n\}$ are to the left of $k-1$. Suppose that \[\pi^\prime=k\ 1\ ^{1}\ 2\ ^{2}\cdots\ k-2 \ ^{k-2}\ k-1.	\]
We claim that $k+1,\ldots,n$ are all to the left of $k-1$ in increasing order and they are inserted into the same site between 1 and $k-1$.

First, the elements $k+1,\ldots,n$ to the left of $k-1$ are increasing to avoid 321.

Second, the elements in $\{k+1,\ldots,n\}$ to the left of $k-1$ must be inserted into the same site.
Suppose that two elements from $k+1,\ldots,n$ with $x<y$ are inserted into the $i$-th and $j$-th site $(i<j)$, respectively. Then $i,x,j,y$ form a copy of 1324, which $\pi$ avoids.

Third, we claim  $k+1,\ldots,n$ are all to the left of $k-1$. Indeed, suppose that there is at least one element in  $\{k+1,\ldots,n\}$ to the right of $k-1$. To avoid 1324, the elements to the right of $k-1$ must be smaller than those to its left. Let $s$ be the largest element to the right of $k-1$. Then $s+1$ and the element immediately to its right together with $s$ form a copy of  $\pattern$. Therefore, there are no elements to the right of $k-1$.
Thus the number of permutations in this subcase is $k-2$.
\end{enumerate}

Thus, the number of permutations in such a case is $$(n-k)+(k-2)+(k-2)=n+k-4.$$

\item $k=n-1$. To avoid 321, the elements $2,\ldots,n-2$ are in increasing order. Suppose that
\[
\pi^{\prime}=n-1\ 1\ ^{1}\ 2\ ^{2}\cdots ^{n-3}\ n-2\ ^{n-2}.
\] Next we will insert $n$ into $\pi^{\prime}$. We see  that all the $n-2$ sites are active. So the number of permutations in this case is $n-2$.

\item $k=n$. To avoid 321, we have $
\pi=n\ 1\ 2\cdots n-1.
$
\end{enumerate}
Therefore, for $n\geq6$,\[
\begin{aligned}
			|F_{n}^{(2)}(321,1324)|=&(n-2)+\sum_{k=3}^{n-2}(n+k-4) +(n-2) +1\\
			=&\dfrac{3}{2}n^{2}-\dfrac{15}{2}n+11.
\end{aligned}
\]

        \vspace{-1cm}

        \ \ \
	\end{proof}

	\begin{proof}[Proof of Theorem \ref{1324}]
		For $n\geq 3$,  by Lemma~\ref{lem:321-213} and Proposition~\ref{prop:321-1324-2}, we have
		\[
		\begin{aligned}
			|F_{n}(321,1324)|=&|F_{n}^{(1)}(321,1324)|+|F_{n}^{(2)}(321,1324)|\\
			=&|F_{n-1}(321,213)|+|F_{n}^{(2)}(321,1324)|
			\\=&\left(n-1\right)+\left(\dfrac{3}{2}n^{2}-\dfrac{15}{2}n+11\right)\\
			=&\dfrac{3}{2}n^{2}-\dfrac{13}{2}n+10.
		\end{aligned}
		\]

        \vspace{-1cm}

        \ \ \
	\end{proof}

\section{Avoiding 321 and two classical patterns of size 4}
 \label{sec3}

\subsection{Enumerating $F_{n}(321,1423,2143)$}
The main result of this subsection is as follows.
\begin{thm}\label{1423-2143}
	For $n\geq 0$,  $|F_{n}(321,1423,2143)|=\binom{n}{2}+1$.
\end{thm}

\begin{prop}
	For $n\geq 2$, $|F_{n}^{(1)}(321,1423,2143)|=n-1.$
\end{prop}
\begin{proof}
	It is routine to prove for $n\leq 3$.
	Let $\pi_2=k.$ We claim that either $k=2$ or $k=3$. Assume that $k\geq 4$. We consider the relative order of 2 and 3. If 2 is to the left of 3, then $1,k,2,3$ form a copy of 1423. If 3 is to the left of 2, the elements $k,3,2$ form a copy of 321. So $k=2$ or $k=3$.
	\begin{itemize}
	\item If $k=3$, to avoid $\pattern$, we have $\pi_3=2$. Moreover,  to avoid 2143, the elements $4,\ldots,n$ appearing after $3\ 2$ must be in increasing order. Therefore, the permutation $ \pi $ takes the form \[
		\pi=1\ 3\ 2\ 4\cdots n.
		\]
	\item If $k=2$, we assume that $\pi=1\ 2\cdots j\ x\cdots$, where $2\leq j\leq n-2$ and $x\geq j+2$.   Note that the case where
	$\pi=1\ 2\cdots n$ is an exceptional case.  We claim that $x=j+2$. To prove this, we consider the order of $j+1$ and $j+2$ in $\pi$. On one hand, if $j+1$ is to the left of $j+2$, then $j,x,j+1,j+2$ form a copy of 1423. On the other hand, if $j+2$ is to the left of $j+1$, then $x,j+2,j+1$ form a copy of 321, which leads to a contradiction. Hence, we conclude that $x=j+2$. Furthermore, $j+1$ must  immediately follow $j+2$. If this is not true, then $j+2$, the element immediately to its right, and $j+1$ together form a copy of  $\pattern$.  Besides, to avoid 2143, the elements  $j+3\cdots n$ appearing after $j+2\ j+1$ must be in increasing order. So \[
		\pi=1\ 2\cdots j\ j+2\ j+1\ j+3\cdots n.
		\]Combining this with the exceptional permutation $\pi=1\ 2\cdots n$, the total number of permutations in this case is $(n-3)+1=n-2$, where $n-3$ accounts for the range of $j$ from 2 to $n-2$.
	\end{itemize}

	Therefore, for $n\geq 4$,  we have $$|F_{n}^{(1)}(321,1423,2143)|=1+(n-2)=n-1.$$

	\vspace{-1cm}

	\ \ \
\end{proof}

\begin{prop}
	For $n\geq 2$, $$|F_{n}^{(2)}(321,1423,2143)|=\binom{n-1}{2}+1.$$
\end{prop}

\begin{proof}
	We can directly check the cases when $n\leq 4$.
	For $n\geq 5$, we let $\pi_1=k$ with $2\leq k\leq n$.  We shall discuss the different possibilities based on the value of $k$.
	\begin{enumerate}[1)]
		\item If $k=2$, then the avoidance of  2143 implies that $3,4,\ldots,n$ to the right of $2\ 1$ must be in increasing order.  So $\pi=2\ 1\ 3\cdots n$.
		\item If $3\leq k\leq n-1$, we need both the elements $2,\ldots,k-1$  and $k+1,\ldots ,n$ to be arranged in increasing order to avoid the patterns 321 and 2143, respectively.  Suppose that  $$\pi^{\prime}=k\ 1\ ^{1}\ 2\ ^{2}\cdots ^{k-3}\ k-2\ ^{k-2}\ k-1 \ ^{k-1}.$$ In this case, we need to insert the elements $k+1,\ldots,n$ into the available sites in increasing order. However, we find that $k+1$ cannot be inserted into the site to the left of $k-2$. If $k+1$ is inserted into each of the sites labeled from 1 to $k-3$, then the elements $1,k+1,k-2,k-1$ would form a copy of 1423, which $\pi$ avoids. Hence, $k+1$ can only be inserted into the sites to the right of $k-2$.
		\begin{enumerate}[a)]
			\item If $k+1$ is inserted into the $(k-2)$-th site, we assume that there are $j$ elements between $k-2$ and $k-1$, where $1\leq j\leq n-k$. The remaining  $n-k-j$ elements need to be inserted into the $(k-1)$-th site in increasing order. The resulting permutation can be represented as follows:
			\[
			\pi=k\ 1\ 2 \cdots k-2\ \underbrace{k+1 \cdots k+j}_j\ k-1\  \underbrace{k+j+1 \cdots n}_{n-k-j}.
			\]For each fixed value of $k$, there is only one valid permutation based on the value of  $j$. As $j$ ranges from 1 to $n-k$, the number of permutations in this subcase is $n-k$.
			\item If $k+1$ is inserted into the $(k-1)$-th site, then \[
			\pi=k\ 1\ 2\ \cdots k-2\ k-1\ k+1\cdots n.
			\]
		\end{enumerate}
		So in this case, the number of permutations is $n-k+1$.
		\item If $k=n$, in order to avoid 321, the elements $2,\ldots,n-1$ must be in increasing order. Thus $\pi=n\ 1\ 2\cdots n-1$.
	\end{enumerate}


	For $n\geq 5$, we have	\[
		|F_{n}^{(2)}(321,1423,2143)|
		=1+\sum_{k=3}^{n-1}(n-k+1)+1
		=\binom{n-1}{2}+1.
	\]

    \vspace{-1cm}
\end{proof}

\begin{proof}[Proof of Theorem \ref{1423-2143}]
We can check the statement directly for $n=0$ and $n=1$.
	For $n\geq 2$, we get
	\[
	\begin{aligned}
		|F_{n}(321,1423,2143)|=&|F_{n}^{(1)}(321,1423,2143)|+|F_{n}^{(2)}(321,1423,2143)|
		\\=&(n-1)+\binom{n-1}{2}+1\\
		=&\binom{n}{2}+1.
	\end{aligned}
	\]

    \vspace{-1cm}

    \ \ \
\end{proof}

\subsection{Enumerating $F_{n}(321,3142,2143)$ }
The main result of this subsection is as follows.
\begin{thm}\label{3142-2143}
	For $n\geq 0$,  $|F_{n}(321,3142,2143)|
	=\binom{n}{2}+1$.
\end{thm}

\begin{proof}
	By Lemma~\ref{lem:classical_pattern}, we have $F_n(321,3142,2143)=S_n(231,321,3142,2143)=S_n(231,321,2143)$. Here we can omit the 3142-pattern since it contains a 231-pattern. Mansour has proved in~\cite[Theorem 3.6(2)]{Mansour2018} that $|S_n(231,321,2143)|=\binom{n}{2}+1$.
\end{proof}

\subsection{Enumerating $F_{n}(321,2143,3124)$}
The main result of this subsection is as follows.
\begin{thm}\label{2143-3124}
	For $n\geq 0$,  $|F_{n}(321,2143,3124)|=\binom{n}{2}+1$.
\end{thm}

\begin{prop}\label{2143-3124-p2=1}
	For $n\geq 2$, $|F_{n}^{(2)}(321,2143,3124)|=n-1.$
\end{prop}

\begin{proof}
We can easily verify the formula for $n\leq 3$.
Now let $n\geq4$. Assume that $\pi_1=k$ with $2\leq k\leq n$.
	\begin{enumerate}[1)]
		\item If $k=2$, avoiding 2143 forces the elements $3,4,\ldots,n$ to the right of $2\ 1$ to be in increasing order. So $\pi=2\ 1\ 3\cdots n $.
		\item For $3\leq k\leq n-1$, we claim that $$\pi=k\ 1\ 2\cdots k-1\ k+1\cdots n.$$
		First,  the elements $2,\ldots,k-1$ (appearing after $k$) and $k+1,\ldots,n$ (appearing after $k$ and 1) must be in increasing order to avoid 321 and 2143, respectively.
		Second, to avoid 3124, the elements $k+1,\ldots,n$ (playing the role of 4) must be to the left of $2,\ldots,k-1$ (playing the role of 2). The form of $\pi$ in this case is  \[
		k\ 1\ k+1\cdots n\ 2\cdots k-1.
		\]
		\item If $k=n$, then  to avoid 321, the elements $2,\ldots,n-1$ must be in increasing order. So $\pi=n\ 1\ 2\cdots n-1$.
	\end{enumerate}
	Fixing $k$, we can uniquely determine a unique  permutation. The number of permutations, which depends on $k$ ranging from 2 to $n$, is $n-1$.

	Therefore, for $n\geq 2$, we have $$|F_{n}^{(2)}(321,2143,3124)|=n-1.$$

    \vspace{-1cm}
\end{proof}

\begin{proof}[Proof of Theorem \ref{2143-3124}]

	The cases when $n<2$ can be easily verified.
		For $n\geq 2$, we have
		\[|F_{n}^{(1)}(321,2143,3124)|=|F_{n-1}(321,2143,3124)|=\binom{n-1}{2}+1.\]
	Thus, it follows from  Proposition~\ref{2143-3124-p2=1} that
		\[
		\begin{aligned}
			|F_{n}(321,2143,3124)|=&\, |F_{n}^{(1)}(321,2143,3124)|+|F_{n}^{(2)}(321,2143,3124)|
			\\=&\Big(\binom{n-1}{2}+1\Big)+(n-1)\\
			=&\binom{n}{2}+1.
		\end{aligned}
		\]

	    \vspace{-1cm}
	\end{proof}

\subsection{Enumerating $F_{n}(321,2143,4123)$}
The main result of this subsection is as follows.
\begin{thm}\label{2143-4123}
	For $n\geq 0$,  $|F_{n}(321,2143,4123)|=\binom{n}{2}+1$.
\end{thm}

	\begin{prop}\label{2143-4123-2}
	For $n\geq 2$, $|F_{n}^{(2)}(321,2143,4123)|=n-1$.
\end{prop}

\begin{proof}
	It is trivial to show that the proposition holds for $n=2, 3$.
	Now we consider $n\geq4$. Let $\pi_1=k$. We first prove that $k<4$. If $k\geq 4$, then both $k-1$ and $k-2$ are located to the right of 1. The avoidance of 321 implies that $k-2$ and $k-1$ are in increasing order. Then $k1(k-2)(k-2)$ would lead to an occurrence of 4123, which is a contradiction. Therefore, we conclude that either $k=2$ or $k=3$.
	\begin{itemize}
		\item If $k=2$, to avoid 2143, the elements $3,4,\ldots,n$ must be arranged in increasing order. Hence, $\pi=2\ 1\ 3\cdots n $.
		\item If $k=3$, to avoid 2143, the elements $4,\ldots,n$ must be in increasing order. Let \[
		\pi^{\prime}=3\ 1\ ^{1}\ 4\ ^{2}\ 5\cdots n\ ^{n-2}.
		\]  It is easy to verify that 2 can be inserted into each of the $n-2$ sites of $\pi^{\prime}$.
		So the number of permutations in this case is $n-2$.
	\end{itemize}
Therefore, for $n\geq 4$,
	$$|F_{n}^{(2)}(321,2143,4123)|=1+(n-2)=n-1.$$

    \vspace{-1cm}

    \ \ \
\end{proof}

\begin{proof}[Proof of Theorem \ref{2143-4123}]
We can easily verify the formula for $n<2$.
	For $n\geq 2$, we get by induction that\[
	|F_{n}^{(1)}(321,2143,4123)|=|F_{n-1}(321,2143,4123)|=\binom{n-1}{2}+1.
	\]
	Thus, we obtain
	\[
	\begin{aligned}
		|F_{n}(321,2143,4123)|=&|F_{n}^{(1)}(321,2143,4123)|+|F_{n}^{(2)}(321,2143,4123)|
		\\=&\Big(\binom{n-1}{2}+1\Big)+(n-1)\\
		=&\binom{n}{2}+1.
	\end{aligned}
	\]

    \vspace{-1cm}

    \ \ \
\end{proof}

\subsection{Enumerating $F_{n}(321,1423,3124)$}
In this subsection and the following two subsections, we will be working with the well-known Fibonacci numbers, see A000045 in OEIS~\cite{oeis}. The Fibonacci numbers, denoted by $F_{n}$, satisfy the initial conditions $F_{0}=F_1=1$ and recurrence relation $F_{n}=F_{n-1}+F_{n-2}$ for $n\geq 2$.

It is worth mentioning that, when counting pattern-avoiding Fishburn permutations involving Fibonacci numbers, we consider the first $k$ elements of $\pi$ (the rearrangement of $1\ 2\cdots\ k$) of some kind of classifications, as a single element playing the role of 1.
In this way, we can  simplify the enumeration process by focusing on permutations of smaller lengths. These permutations can then be transformed into the enumeration problem that has been previously established, allowing for easier enumeration and induction.

The main result of this subsection is as follows.
\begin{thm}\label{1423-3124}
	For $n\geq 4$,  $$|F_{n}(321,1423,3124)|=F_{n}+2.$$
\end{thm}

 \begin{lem}\label{lem:321-312}
 	For $n\geq 2$, we get $$|F_{n}(321,312)|=F_{n}.$$
 \end{lem}

\begin{proof}
By Lemma~\ref{lem:classical_pattern}, we have $|F_{n}(321,312)|=|S_n(231,321,312)|$. In \cite[Lemma 6 (a)]{Simion1985}, the set $S_n(231,321,312)$ has been proved to be enumerated by the Fibonacci numbers $F_n$. So we get $|F_{n}(321,312)|=F_{n}$.
\end{proof}

\begin{prop}
	For $n\geq 4$, $$|F_{n}^{(2)}(321,1423,3124)|=F_{n-2}+2.$$
\end{prop}
\begin{proof}
	Suppose $\pi_1=k.$ We claim the value of $k$ can be chosen from $\{2,3,n\}$. If not, we assume that $4\leq k\leq n-1$. In this case, there are at least two elements less than $k$ to the right of $k\ 1$ and they cannot contain descents to avoid 321. Let  $$\pi^{\prime}=k\ 1\ ^{1}\ 2\ ^{2} \cdots k-1\ ^{k-1}\cdots.$$ We consider which site after 1 we can insert $n$ into. If $n$ is inserted into the first site, then $1,n,2,k-1$ form a copy of 1423.
	   Whereas if $n$ is inserted into some site to the right of 2, then $k,1,2,n$ form a copy of 3124, which is a contradiction. So either $k<4$ or $k=n$.
	\begin{enumerate}[1)]
		\item $k=2$. The elements 2,1 can be seen as one element playing the role of 1. { For the same reason as in the later proof of Theorem \ref{1423-3124},} the number of permutations in this case is \[
		|F_{n-2}(321,312)|=F_{n-2}.
		\]
		\item $k=3$. First we have $\pi_{n}=2$. If there is an element $x\geq4$ to the right of 2, then $3,1,2,x$ form a copy of 3124. Second, avoiding 321 forces $4,\ldots,n$ to the left of 2 must be in increasing order. So\[
		\pi=3\ 1\ 4\cdots n\ 2.
		\]\item $k=n$. To avoid 321, the elements $2,\ldots, n-1$ must be in increasing order. So \[
		\pi =n\ 1\ 2\cdots n-1.
		\]
	\end{enumerate}
	Thus, $$|F_{n}^{(2)}(321,1423,3124)|=F_{n-2}+1+1=F_{n-2}+2.$$

    \vspace{-1cm}

    \ \ \
\end{proof}

\begin{proof}[Proof of Theorem \ref{1423-3124}]
	Let $\pi\in F_{n}^{(1)}(321,1423,3124)$ with $\pi=1\oplus\tau$. We claim $\tau\in F_{n-1}(321,312)$. Indeed, the  avoidance of 1423 implies that $\tau$ avoids 312.  This is because if $ \tau $ contained a 312-pattern, then when we add the element 1 as the first element of $ \tau $, it would create a 1423-pattern. Similarly, $\pi$ avoiding 3124 means $\tau$ avoiding 3124 since the element 1 cannot be involved in every copy of 3124-pattern. Moreover, if there is no 312-pattern in $ \tau $, it also implies that there is no 3124-pattern in $ \tau $ since a 3124-pattern contains a 312-pattern. Therefore,\[
|F_{n}^{(1)}(321,1423,3124)|=|F_{n-1}(321,312)|=F_{n-1}.
\]
	For $n\geq 4$, we have
	\[
	\begin{aligned}
		|F_{n}(321,1423,3124)|=&|F_{n}^{(1)}(321,1423,3124)|+|F_{n}^{(2)}(321,1423,3124)|
		\\=&F_{n-1}+(F_{n-2}+2)\\
		=&F_{n}+2.
	\end{aligned}
	\]

    \vspace{-1cm}
\end{proof}

\subsection{Enumerating $F_{n}(321,1423,4123)$}
The main result of this subsection is as follows.
\begin{thm}\label{1423-4123}
	For $n\geq 1$,  $$|F_{n}(321,1423,4123)|=F_{n+1}-1.$$
\end{thm}

\begin{prop}
	For $n\geq 4$, \begin{equation}\label{(1423,4123)-1}
		|F_{n}^{(1)}(321,1423,4123)|=F_{n-1}.
	\end{equation}
\end{prop}

\begin{proof}
	Suppose that $\pi=1\oplus\tau\in F_{n}^{(1)}(321,1423,4123)$. We claim that $\tau\in F_{n-1}(321,312)$. The proof of Theorem~\ref{1423-3124} establishes that if $\pi$ avoids the pattern 1423, then $\tau$ avoids 312. Moreover, if $\tau$ avoids 312, it must also avoid 4123 since a 4123-pattern contains a 312-pattern.

	Therefore, by Lemma~\ref{lem:321-312}, we have\[
	|F_{n}^{(1)}(321,1423,4123)|=|F_{n-1}(321,312)|=F_{n-1}.
	\]

	\vspace{-1cm}
\end{proof}

\begin{prop}
	For $n\geq 4$, $$|F_{n}^{(2)}(321,1423,4123)|=F_{n}-1.$$
\end{prop}
\begin{proof}
	We claim that $\pi_{1}=2$ or $\pi_{1}=3$ to avoid patterns 321 and 4123, which has been proved in the proof of Proposition \ref{2143-4123-2}.
	 Next we consider two possible values of $\pi_1$.
\begin{enumerate}[1)]
\item	If $\pi_{1}=2$, we can treat the elements 2 and 1 together as a single element, which plays the role of 1. The number of permutations in this case is \[
	|F_{n-1}^{(1)}(321,1423,4123)|=F_{n-2}.
	\]

\item	If $\pi_{1}=3$, we consider the position of 2.
	\begin{enumerate}[a)]
		\item $\pi_3=2$. We can treat three elements $3\ 1\ 2$ as one element playing the role of 1. Thus,\[
		|F_{n-2}^{(1)}(321,1423,4123)|=F_{n-3}.
		\]
		\item $\pi_3\neq 2$. We claim 	\[
		\pi=3\ 1\ 4\cdots k\ 2\cdots
		\]with $4\leq k\leq n$. First we shall prove that $\pi_3=4$ by contradiction. Suppose that $\pi_3>4$. We consider the value of $\pi_4$. On the one hand, if $\pi_4=2$, the element 4 is to the right of 2 and $1,\pi_3,2,4$ form a copy of 1423. On the other hand, if $\pi_4\neq 2$, then  to avoid 321, $\pi_3$ and $\pi_4$ must be in increasing order since 2 is to the right of them. Under the assumption of $\pi_3>4$, we have $\pi_3-1$ must be to the right of $\pi_4$ and $\pi_3,\pi_4,\pi_3-1$ form a copy of  $\pattern$, which is a contradiction.

		Second, suppose that there are $\ell$ elements between 1 and 2. If $\ell=1$, then the element is~4. If $\ell>1$, we denote them by $a_1,a_2,\ldots,a_\ell$ with $a_1=4$. Next we show that $a_{i+1}=a_{i}+1$ for $1\leq i\leq \ell-1$. The elements $a_1,a_2,\ldots,a_\ell$ to the left of 2 must be increasing to avoid 321. Additionally, if there exists $j\in [\ell]$ such that $a_{j}\neq a_{j-1}+1$, then $a_{j}-1$ is to the right of 2. Thus $1,a_{j},2,a_{j}-1$ form a copy of 1423. Let $a_\ell=k$. We substitute $a_1,a_2,\ldots,a_\ell$ with $4,5,\ldots,k$, where $k$ ranges from 4 to $n$.

		The first $k$ elements $3,1,4,\ldots,k,2$ can  be seen as one element to play the role of 1. So 	\[
		|F_{n-(k-1)}^{(1)}(321,1423,4123)|=F_{n-k}.
		\]
	\end{enumerate}
\end{enumerate}

	Therefore,  we have
	\begin{align*}
		|F_{n}^{(2)}(321,1423,4123)|=&F_{n-2}+F_{n-3}+\sum_{k=4}^{n}F_{n-k}\\
		=&\sum_{k=0}^{n-2}F_{k} = F_{n}-1.
	\end{align*}
    \vspace{-1cm}

    \ \ \
\end{proof}

\begin{proof}[Proof of Theorem \ref{1423-4123}]
	We can check the statement directly for $n\leq 3$.
	For $n\geq 4$, we have
	\[
	\begin{aligned}
		|F_{n}(321,1423,4123)|=&|F_{n}^{(1)}(321,1423,4123)|+|F_{n}^{(2)}(321,1423,4123)|
		\\=&F_{n-1}+(F_{n}-1)\\
		=&F_{n+1}-1.
	\end{aligned}
	\]
    \vspace{-1cm}

    \ \ \
\end{proof}

\subsection{Enumerating $F_{n}(321,3124,4123)$}
The main result of this subsection is as follows.
\begin{thm}\label{3124-4123}
	For $n\geq 1$,  $$|F_{n}(321,3124,4123)|=F_{n+1}-1.$$
\end{thm}

\begin{prop}
	For $n\geq 4$,
	\begin{equation}
		|F_{n}^{(2)}(321,3124,4123)|=F_{n-1}.
	\end{equation}
\end{prop}
\begin{proof}
	We claim that $2\leq \pi_{1}\leq 3$. If not, suppose that $\pi_{1}\geq 4$. We consider the order of 2 and 3.
	If $2\ 3$ are increasing, then $\pi_{1},1,2,3$ form a copy of 4123. However, if they are in descending order, then $\pi_{1},3,2$ form a copy of 321. So $\pi_1=2$ or $\pi_1=3$. As a result, the number of permutations can be classified into two cases based on the value of $\pi_1$.

	If $\pi_{1}=2$, we can treat $2\ 1$ as one element playing the role of 1. The number of permutations in this case can be reduced to  \[
	|F_{n-1}^{(1)}(321,3124,4123)|=F_{n-1}-1.
	\]

	If $\pi_{1}=3$, we have $ \pi=3\ 1\ 4\cdots n\ 2. $ First we shall show $\pi_{n}=2$. If  $\pi_{n}\neq 2$, then $3,1,2,\pi_{n}$ form a copy of 3124. Second, to avoid 321, the elements $4,\ldots,n$ appearing before 2 must be in increasing order.

	Therefore, we have
	\[
	|F_{n}^{(2)}(321,3124,4123)|=(F_{n-1}-1)+1=F_{n-1}.
	\]
    \vspace{-1cm}

    \ \ \
\end{proof}

\begin{proof}[Proof of Theorem \ref{3124-4123}]
	We can check the formula directly for $n\leq 3$.
	By induction, we obtain
	$$|F_{n}^{(1)}(321,3124,4123)|=|F_{n-1}(321,3124,4123)|=F_{n}-1.$$
	For $n\geq 4$, we have
	\begin{align*}
		|F_{n}(321,3124,4123)|=&|F_{n}^{(1)}(321,3124,4123)|+|F_{n}^{(2)}(321,3124,4123)|
		\\=&(F_{n}-1)+F_{n-1}\\
		=&F_{n+1}-1.
	\end{align*}

    \vspace{-1cm}
\end{proof}


\section{Avoiding 321 and a classical pattern of size 5}
\label{sec4}
\subsection{Enumerating $F_n(321,14253)$}
The main result of this subsection is as follows.
\begin{thm}\label{321-14253}
	For $n\geq 1$, we have
	\[
  		|F_n(321,14253)|=2^n-\binom{n}{2}-1.
	\]
\end{thm}

In order to prove Theorem~\ref{321-14253}, we first enumerate $F_n(321,3142)$.

\begin{lem}\label{lem3142}
	For $n\geq 1$, we have
	\begin{equation}\label{321-3142}
		|F_n(321,3142)|=2^{n-1}.
	\end{equation}
\end{lem}

\begin{proof}
By Lemma~\ref{lem:classical_pattern}, we have $F_n(321,3142)=S_n(231,321,3142)$. Actually the avoidance of 3142 can be omitted since a 3142-pattern contains a 231-pattern. Simion and Schmidt have proved in~\cite[Lemma 5 (a)]{Simion1985} that $|S_n(321,231)|=2^{n-1}$.
Therefore, we get $|F_n(321,3142)|=2^{n-1}$.
\end{proof}

It is worth mentioning that there exists a bijection between the sets of left-to-right maxima of $F_n(321,3142)$ and the subsets of $[n]$ containing $n$. Recall that $\pi_i$ is a left-to-right maximum of $\pi$ if $\pi_i$ is greater that all the entries to its left. For instance, the left-to-right maxima of permutation $$\pi=3\ 1\ 2\ 4\ 7\ 5\ 6$$ are 3,4,7. Conversely, given the left-to-right maxima, we can uniquely determine the entire permutation. For instance, as the left-to-right maxima are 3,4,7, we  deduce that $\pi_1=3$ and $\pi_2=1$. To avoid 3142, the element 2 is to the left of 4. Additionally, since the elements 5,6 are not left-to-right maxima, they are to the right of 7. Thus we have determined $\pi$.

\begin{proof}[Proof of Theorem \ref{321-14253}]
We can directly check that Theorem \ref{321-14253} holds for $n\leq 5$. For $n\geq 6$, we shall classify the permutations based on the position at which the element 1 appears.
Let $\pi$ be a permutation in $F_n(321,14253)$.

 If $\pi_1=1$, then $\pi=1\oplus\tau$ with $\tau \in F_{n-1}(321,3142)$. By equality \eqref{321-3142}, the number of permutations is $|F_{n-1}(321,3142)|=2^{n-2}$.

If $\pi_2=1$, we divide our discussion into five cases according to the values of $\pi_1$.
\begin{enumerate}[1)]
	\item $\pi_1=2$. Then $\pi=2\ 1\oplus\tau$ with $\tau\in F_{n-2}(321,3142)$. Hence, the number of permutations in this case is $|F_{n-2}(321,3142)|=2^{n-3}$.

	\item $\pi_1=k$ with $3\leq k\leq n-3$.
	To avoid 321,  it is necessary for the elements   $2,\ldots,k-1$ that follow after  $k$  to be in increasing order. We next consider whether there are elements to the left of $k-1$ or not.
\begin{enumerate}[a)]
	\item No elements in $k+1,\ldots,n$  are inserted to the left of $k-1$. Thus, the first $k$ elements of $\pi$ are $k,1,2\cdots,k-1$ and $$\pi=k\ 1\ 2\cdots k-1\oplus \tau$$ with $\tau\in F_{n-k}(321,3142)$. The number of permutations in this case is $$|F_{n-k}(321,3142)|=2^{n-k-1}.$$
	\item There is at least one element in the range $k+1\ldots,n$ that is to the left of $k-1$.
	We first claim that  if two or more elements in the set $\{k+1\ldots,n\}$ are positioned to the left of $k-1$, they must be arranged in increasing order and inserted into exactly one of the available sites.
	If $k=3$,  there is only one available space between 1 and 2.
	If $k\geq 4$, there are at least two spaces between 1 and $k-1$.
	Let \[
		\pi^{\prime}=k\ 1\ ^{1}\ 2\ ^{2}\cdots i\ ^{i}\ i+1\cdots ^{\, k-2}\ k-1\ ^{k-1}.
		\]
	Assume two elements, denoted $x$ and $y$, from the set $\{k+1,\ldots,n\}$ are inserted into two distinct sites ranging from 1 to $k-1$. To avoid the pattern 321, it is necessary that $x<y$. Let $i$ be an element between $x$ and $y$ with $2\leq i\leq k-3$. Then  $1,x,i,y,k-1$ forms a copy of the pattern 14253, which is forbidden in $\pi$. This proves our claim.

We next prove that there are precisely  $2^{n-k}-1$ ways to divide $k+1,\ldots,n$ into the $i$-th and the $(k-1)$-th sites while ensuring that  $i$-th site is non-empty.
	{We make the following classification based on the number of elements in the $i$-th site.}
	\begin{itemize}
\item Only one element $\ell \in\{k+1,\ldots,n\}$ is inserted into the $i$-th site.
In order to avoid 321, the elements $k+1 \cdots \ell-1$ are in increasing order. Additionally, $1, \ell, k-1, \ell-1$ form a copy of 1423. To avoid the pattern 14253, the elements $\ell+1,\ldots,n$, representing  the role of 5, are positioned to the right of $\ell-1$. The first $\ell$ elements of the permutation are $$k,1,2 \cdots i, \ell, i+1 \cdots k-1, k+1 \cdots \ell-1$$ and  they can be considered as the smallest element.
Note that when $\ell=k+1$, we have $\pi_\ell=k-1$.
 Thus, \[
		|F_{n-\ell}(321,3142)|=\left\{\begin{array}{cc}
			2^{n-\ell-1}, & k+1 \leq \ell \leq n-1, \\
			1, & \ell=n .
		\end{array}\right.
		\]
Thus, the total number of permutations with one element inserted into the $i$-th site is\[
\sum_{\ell=k+1}^{n-1}2^{n-\ell-1}+1=2^{n-k-1}.
\]
\item More than one elements $a_1,a_2,\ldots,a_s\in\{k+1,\ldots,n\}$ with $2\leq s\leq n-k$ are inserted into the $i$-th site. Then 321-avoiding implies that $$a_1<a_2<\cdots<a_s.$$
The $\pattern$-avoiding condition forces $a_1=k+1, a_2=k+2$ and so forth until $a_{s-1}=k+s-1$, and $a_s\geq k+s$.
Observe that removing $k+1, k+2, \ldots, k+s-1$ does not affect our enumeration.
Similar to the above paragraph, the number of permutations with more than one elements inserted into $i$-th site is
\[
\sum_{s=2}^{n-k} 2^{n-k-s} =2^{n-k-1}-1.
\]
\end{itemize}
Consequently, as $i$ has $k-2$ choices, the number of permutations in this subcase is  $$(k-2)\left(2^{n-k-1}+(2^{n-k-1}-1)\right) = (k-2) \left(2^{n-k}-1\right).$$

	\end{enumerate}
\item $\pi_1=n-2$. The elements $2,\ldots,n-3$ that appear after $n-2$ must be in increasing  order. Let $$\pi^{\prime}=n-2\ 1\ 2\cdots i\ ^{1}\ i+1\cdots n-3\ ^{2}$$ where $1\leq i\leq n-3$.
We can insert $n-1$ and $n$ into $\pi^{\prime}$ in the following ways.
\begin{itemize}
	\item $n-1$ is in the first site and $n$ is in the second site.
	\item $n$ is in the first site and $n-1$ is in the second site.
	\item $n-1$ and $n$ are both in the first site in increasing order.
	\item $n-1$ and $n$ are both in the second site.
\end{itemize}
 The first three subcases are  associated with  $1\leq i\leq n-4$. In the last subcase, $n-1$ and $n$ can be inserted  either in increasing  or descending  order.
Thus the number of permutations when $\pi_1=n-2$ is $3(n-4)+2=3n-10$.
\item $\pi_1=n-1$. Then $2,\ldots,n-2$ appearing after $n-1$ must be in increasing order. Let $\pi^{\prime}=n-1\ 1\ ^{1}\ 2\ ^{2}\cdots\ ^{n-3}\ n-2\ ^{n-2}$. We should insert $n$ into $\pi^{\prime}$ to get $\pi$. Actually, all the $n-2$ sites are available to insert $n$. Thus the number of permutations in this case is $n-2$.
\item $\pi_1=n$. To avoid 321, the elements $2,\ldots,n-1$ appearing after $n$ must be in increasing order. So $\pi=n\ 1\ 2\cdots n-1$.

\end{enumerate}
Therefore, for $n\geq 6$, it follows that
\begin{align*}
	F_n(321,14253)&=2^{n-2}+2^{n-3}+\sum_{k=3}^{n-3} \left( 2^{n-k-1}+(k-2)(2^{n-k}-1) \right) +(3n-10)+(n-2)+1\\&=2^n-\binom{n}{2}-1.
\end{align*}
This completes the proof.
\end{proof}

\subsection{Enumerating $F_n(321,21354)$}
The main result of this subsection is as follows.
\begin{thm}\label{321-21354}
	For $n\geq 1$, we have
	\begin{equation}\label{eq21354}
	|F_n(321,21354)|=2^n-\binom{n}{2}-1.
	\end{equation}
\end{thm}

For our convenience, we  let   $F_n^{<k12\cdots\ell>}(321,21354)\ (1\leq \ell\leq k-1)$  denote the set of  Fishburn permutations of length $n$ that avoid  both 321 and 21354, and begin with $k\ 1 \ 2\cdots\ell$.
Let $F_n^{<k12\cdots\bar{\ell}>}(321,21354)$ denote the set of those permutations that begin with $k\ 1 \ 2\cdots\ell-1$ but not $k\ 1 \ 2\cdots\ell$.

\begin{proof}[Proof of Theorem \ref{321-21354}]
Our proof is  by induction. One can directly check the equality \eqref{eq21354}  holds for $n\leq 3$.
Suppose that \eqref{eq21354} holds for $n-1$. We shall prove it holds for $n$ as well.
Let $\pi\in F_n(321,21354)$. If $\pi_1=1$, then $\pi=1\oplus\tau$ with $\tau\in F_{n-1}(321,21354)$. Thus it follows from the induction hypothesis that \[
|F_n^{(1)}(321,21354)|=|F_{n-1}(321,21354)|=2^{n-1}-\binom{n-1}{2}-1.
\]
Therefore it suffices to prove
the following equality
\begin{align*}
	|F_n^{(2)}(321,21354)|&=|F_{n}(321,21354)|-|F_n^{(1)}(321,21354)|\\[5pt]
	&= \left( 2^{n}-\binom{n}{2}-1 \right) - \left( 2^{n-1}-\binom{n-1}{2}-1 \right) \\[5pt]
	&=2^{n-1}-n+1.
\end{align*}
For the enumeration of $F_n^{(2)}(321,21354)$, our  classification is  according to the value of $\pi_1$.
\begin{enumerate}[1)]
	\item $\pi_1=2$. Then $\pi=2\ 1\oplus\tau$ with $\tau \in F_{n-2}(321,132)$. By Lemma~\ref{lem:321-132}, the number of permutations in this case is $|F_{n-2}(321,132)|=n-2$.
	\item $\pi_1=3$. Next we proceed to consider the position of 2.
	\begin{enumerate}[a)]
		\item The element 2 is immediately to the right of 1. We have $$\pi=3\ 1\ 2\oplus\tau,$$ where $\tau\in F_{n-3}(321,132)$. The number of permutations in this subcase is $$|F_{n-3}(321,132)|=n-3.$$
		\item There is only one element $m$ between 1 and 2 with $4\leq m\leq n$.
		\begin{itemize}
			\item $m=4$. Then $5,\ldots,n$ appearing after $3\ 1\ 4$ must be increasing  to avoid 21354. So $\pi=3\ 1\ 4\ 2\ 5\cdots n$.
			\item $m\geq 5$. First, to avoid 321, the sequence $4,\ldots,m-1$ appearing after $m$ must be increasing.
			Second, avoiding 21354 implies that the sequence $m+1,\ldots,n$ appearing after $3\ 1\ m$ must be increasing and  no elements in $\{m+1,\ldots,n\}$ can be inserted into the sites between 4 and $m-1$. The former is easy to see. As for the latter, if $x\in\{m+1,\ldots,n\}$ is inserted into one site between 4 and $m-1$, then $3\ 1\ 4\ x\ m-1$ form a copy of 21354. Hence, we can consider $4\cdots m-1$ as one element and insert it into the $n-m+1$ available sites that the elements $m+1,\ldots,n$ formed. There are $n-m+1$ ways to insert $4\cdots m-1$ to get $\pi$.
		\end{itemize}
	Thus, the number of permutations  in this subcase is\[
	1+\sum_{m=5}^{n}(n-m+1).
	\]
		\item There are at least two elements between 1 and 2. Then we have
		\[
		\pi=3\ 1\ 4\cdots m\ 2\ m+1\cdots n,
		\] where $5\leq m\leq n$.

		First, the elements $\pi_3$ and $\pi_4$ appearing before 2 must be increasing to avoid 321.
		Second, we have $\pi_3=4.$ If not, $\pi_3-1$ appears to the right of 2 and $\pi_3,\pi_4,\pi_3-1$ form a copy of  $\pattern$.
		Third, to avoid 21354, the elements $5,\ldots,m$ to the right of $3\ 1\ 4$ are in increasing order. Let the largest element to the left of 2 be $m$. So $\pi$ is of the above form.

		Therefore, the number of permutations in this subcase is $n-4$.
	\end{enumerate}
It is worth mentioning that the number of those permutations with $\pi_1=3$ and $\pi_2=1$ but $\pi_3\neq 2$ is\[
1+\sum_{m=5}^{n}(n-m+1)+(n-4)=\sum_{m=1}^{n-3}m=\binom{n-2}{2}.
\]
Thus, we have\[
|F_n^{<31\bar{2}>}(321,21354)|=\binom{n-2}{2}.
\]

\item $\pi_1=k$ with $4\leq k\leq n-1$. The number of permutations in this case is\[
|F_n^{<k12\cdots k-1>}(321,21354)|+\sum_{i=2}^{k-1}|F_n^{<k12\cdots \bar{i}>}(321,21354)|.
\]

	For every permutation $\pi$ that begins with $k\ 1\ 2\cdots k-1$, we have $\pi=k\ 1\ 2\cdots k-1\oplus \tau$ with $\tau\in F_{n-k}(321,132)$. Thus\[
	|F_n^{<k12\cdots k-1>}(321,21354)|=|F_{n-k}(321,132)|=n-k.
	\]
	In the following part, we discuss the enumeration of $F_n^{<k1\bar{2}>}(321,21354)$ for $k\geq 4$. The enumeration of $F_n^{<k12\cdots\bar{\ell}>}(321,21354)$ for $3\leq \ell\leq k-1$ can be reduced to this form. We classify according to the number of elements between 1 and 2.
	\begin{enumerate}[a)]
		\item There is only one element $m$ between 1 and 2. Let $\pi=k\ 1\ m\ 2\cdots$. Then to avoid 321, we have $m>k$ because $k$ appears before $m\ 2$. Here we consider $m\geq k+2$. The case of $m=k+1$ is discussed later in Subcase b). Next we need to consider the arrangement of the remaining elements $\{3,\ldots,k-1\}\cup\{k+1,\ldots,m-1\}\cup\{m+1,\ldots,n\}$.
		First, the sequence $3,\ldots,k-1$, the sequence $k+1,\ldots,m-1$, and the sequence $m+1,\ldots,n$  must be increasing to avoid the pattern 321, 321, and 21354, respectively.
		Second,  the sequence $k+1,\ldots,m-1$ must be to the right of the increasing sequence  $3\cdots k-1$ to avoid 321.
		Third, no elements from $m+1,\ldots,n$ can be inserted into the sites between $k+1$ and $m-1$. If $x\in \{m+1,\ldots,n\} $ is inserted into one of these sites, then $k,1,k+1,x,m-1$ form a copy of 21354. Hence, we can consider $k+1,\ldots,m-1$ as one element and thus consider $\{3,\ldots,k-1\}\cup\{k+1,\ldots,m-1\}\cup\{m+1,\ldots,n\}$ as two increasing sequences: one of length $k-2$ composed of elements from the first two sets, and the other of length $n-m$ composed of elements from $ \{m+1,\ldots,n\} $. By treating this as a counting problem, we can consider the number of permutations of the multiset $\{(k-2)\cdot A,(n-m)\cdot B\}$, which is $\binom{n-m+k-2}{k-2}$.

		\item There are at least two elements between 1 and 2. Then we claim that the elements between 1 and 2 are the elements $k+1,\ldots,m$ in increasing order with $k+1\leq m\leq n$ (notice that $m=k+1$ is discussed here). Namely $\pi$ is of the form $\pi=k\ 1\ k+1\cdots m\ 2\cdots$, where $m\geq k+1$. First, the elements between 1 and 2 must be greater than $k$ and they are increasing to avoid 321. Second, if $\pi_3\neq k+1$, the elements $\pi_3,\pi_4,\pi_3-1$ form a copy of  $\pattern$. So $\pi_3=k+1$.
		Third, from $\pi_4$ onwards, the elements follow a consecutive increasing sequence $\pi_4=k+2,\pi_5=k+3$ and so forth. If  at some point $\pi_{i}\neq\pi_{i-1}+1$, then $k,1,k+1,\pi_i,\pi_i-1$ form a copy of 21354.
		Similarly to the above, the sequence $3,\ldots,k-1$ and the sequence $m+1,\ldots,n$ are two increasing sequences of lengths $k-3$ and $n-m$, respectively. Thus the counting can be turned into counting the number of permutations of the multiset $\{(k-3)\cdot A,(n-m)\cdot B\}$, which is $\binom{n-m+k-3}{k-3}$.
	\end{enumerate}

	Thus,
	\allowdisplaybreaks
	\begin{align*}
		|F_n^{<k1\bar{2}>}(321,21354)|&=\sum_{m=k+2}^{n}\binom{n-m+k-2}{k-2}+\sum_{m=k+1}^{n}\binom{n-m+k-3}{k-3}\\
		&=\sum_{m=k-2}^{n-4}\binom{m}{k-2}+\sum_{m=k-3}^{n-4}\binom{m}{k-3}\\
		&=\binom{n-3}{k-1}+\binom{n-3}{k-2}
		=\binom{n-2}{k-1},
	\end{align*}
	where the third equality follows from the hockey stick identity.

For the enumeration of $|F_n^{<k12\cdots \bar{i}>}(321,21354)|$ with $i\geq 3$, we can consider $1\ 2\cdots i-1$ as one element playing the role of 1,  consider $i$ as 2 and consider $k$ as $k-i+2$. Thus, we have \[
|F_n^{<k12\cdots \bar{i}>}(321,21354)|=|F_{n-i+2}^{<(k-i+2)1\bar{2}>}(321,21354)|=\binom{n-i}{k-i+1}.
\]
Therefore, we obtain
\begin{align*}
	&|F_n^{<k12\cdots k-1>}(321,21354)|+\sum_{i=2}^{k-1}|F_n^{<k12\cdots \bar{i}>}(321,21354)|\\
	=&(n-k)+\binom{n-2}{k-1}+\sum_{i=3}^{k-1}\binom{n-i}{k-i+1}
	=\sum_{i=2}^{k}\binom{n-i}{k-i+1}.
\end{align*}
\item $\pi_1=n$. Avoiding 321 indicates that $2,\ldots,n-1$ to the right of $n$ must be increasing. So\[
\pi=n\ 1\ 2\cdots n-1.
\]
\end{enumerate}
Thus, we have
\allowdisplaybreaks
\begin{align*}
	|F_n^{(2)}(321,21354)|=&|F_n^{<21>}(321,21354)|+|F_n^{<312>}(321,21354)|+|F_n^{<31\bar{2}>}(321,21354)|\\
	&\quad +\sum_{k=4}^{n-1}\left(|F_n^{<k12\cdots k-1>}(321,21354)|+\sum_{i=2}^{k-1}|F_n^{<k12\cdots \bar{i}>}(321,21354)|\right)\\
	&\quad +|F_n^{<n1>}(321,21354)|\\
	=&(n-2)+(n-3)+\binom{n-2}{2}+\sum_{k=4}^{n-1}\sum_{i=2}^{k}\binom{n-i}{k-i+1}+1\\
	=&(n-2)+\sum_{k=3}^{n-1}\sum_{i=2}^{k}\binom{n-i}{n-k-1}+1\\
	=&(n-2)+\sum_{k=3}^{n-1}\Big(\binom{n-1}{k-1}-1 \Big)+1\\
	=&2^{n-1}-n+1,
\end{align*}
where the second-to-last equality follows from the hockey stick identity.

We have thus proved Theorem~\ref{321-21354} by induction.
\end{proof}

\subsection{Enumerating $F_n(321,31452)$}
The following three enumeration results are about \emph{Pell numbers}, denoted by $P_n$ with the initial condition $P_0=0, P_1=1$ and recurrence relation $P_n=2P_{n-1}+P_{n-2}$ for $n\geq 2$; see A000129 in OEIS~\cite{oeis}.
\begin{thm}\label{321-31452}
	For $n\geq 1$,\[
	|F_n(321,31452)|=\dfrac{P_n+P_{n-1}+1}{2}.
	\]
\end{thm}


\begin{proof}
	For our convenience, we denote $\frac{P_n+P_{n-1}+1}{2}$ by $Q_n$ for $n\geq1$ and define $Q_0=1$.

	Our proof is by induction on $n$ at the same time, with easily checkable base cases, for
	\begin{equation}\label{31452}
		|F^{(1)}_n(321,31452)|=Q_{n-1}
	\end{equation}
and \begin{equation}\label{31452-2}
	|F_n^{(2)}(321,31452)|=P_{n-1}.
\end{equation}
Suppose that the above two equalities hold for positive integers less than $n$. We shall prove them hold for $n$ as well. We classify according to the position of 1.
\begin{enumerate}[1)]
	\item $\pi=1\oplus\tau$ with $\tau\in F_{n-1}(321,31452)$. By the induction hypothesis, the number of permutations is\[
	|F_{n-1}(321,31452)|=Q_{n-1}.
	\]
	\item $\pi=2\ 1\oplus\tau$ with $\tau\in F_{n-2}(321,31452)$. The number of permutations in this case is \[
	|F_{n-2}(321,31452)|=Q_{n-2}.
	\]
	\item $\pi=k\ 1\cdots$ with $3\leq k\leq n$. The elements $2,\ldots,k-1$ are increasing to avoid 321. Next we classify according to whether there are elements in $k+1,\ldots,n$ to the left of $k-1$ or not.
	\begin{enumerate}[a)]
		\item  $\pi=k\ 1\cdots k-1\oplus\tau$ with $\tau\in F_{n-k}(321,31452)$. The number of permutations in this case is \[
		|F_{n-k}(321,31452)|=Q_{n-k}.
		\]
		\item There are some elements to the left of $k-1$.

		We claim that only one element from $\{k+1,\ldots,n\}$ can be inserted into exactly one of the $k-2$ sites between 1 and $k-1$. Suppose there are two elements $x$ and $y$, from the range $k+1$ to $n$, that are located to the left of $k-1$. If $x>y$, then $x,y,k-1$ form a copy of 321. Conversely, if $x<y$, then $k,1,x,y,k-1$ form a copy of $31452$, a contradiction.

		Suppose that $x$ is inserted into the site between $i$ and $i+1$ with $1\leq i\leq k-2$ and $k+1\leq x\leq n$. Then \[
		\pi=k\ 1\ 2\cdots i\ x\ i+1\cdots k-1\cdots.
		\]
		Since $k,1,2,\ldots,i$ cannot create a copy of  $\pattern$, 321, 31452 when followed by elements to their right, we can treat the enumeration in this case as the enumeration of permutations of length $n-k+1$ of the following form
		\[
		\pi^{\prime}=x-k+1\ 1\cdots.
		\] We just remove $k,1,2,\ldots,i$ and consider $i+1\cdots k-1$ as a whole part to play the role of the smallest element 1. Thus the enumeration for a fixed $i$ can be transformed into counting the number of permutations of length $n-k+1$ with the smallest element 1 in the second site. Therefore, by equality \eqref{31452-2}, the number of permutations in this case is \[
		(k-2)|F_{n-k+1}^{(2)}(321,31452)|=(k-2)P_{n-k}.
		\]
	\end{enumerate}
\end{enumerate}

  Thus, \begin{equation}\label{q+kp}
  	|F_{n}^{(2)}(321,31452)|=Q_{n-2}+\sum_{k=3}^{n}\Big((k-2)P_{n-k}+Q_{n-k}\Big)=P_{n-1},
  	\end{equation}
  where the last equality shall be proved in Appendix.

  Therefore, we have
  $$
  	|F_{n}(321,31452)|=|F_{n}^{(1)}(321,31452)|+|F_{n}^{(2)}(321,31452)|
  	=Q_{n-1}+P_{n-1}=Q_{n}.$$
  This completes the proof.
\end{proof}

\subsection{Enumerating $F_{n}(321,31524)$}
\begin{thm}\label{31524}
	For $n\geq 1$,\[
	|F_n(321,31524)|=\dfrac{P_n+P_{n-1}+1}{2}.
	\]
\end{thm}

\begin{proof}
	Let $Q_n:= \frac{P_n+P_{n-1}+1}{2}.$  We prove by induction on $n$ that $|F_n(321,31524)|=Q_n.$  The base cases are easy to check. Assuming that this equality holds for positive integers less than $n$, we shall prove it holds for $n$ as well.
	\begin{enumerate}[1)]
		\item $\pi=1\oplus\tau$ with $\tau\in F_{n-1}(321,31524)$. By the induction hypothesis, the number of permutations~is\[
		|F_{n-1}(321,31524)|=Q_{n-1}.
		\]
		\item $\pi=2\ 1\oplus\tau$ with $\tau\in F_{n-2}(321,31524)$. The number of permutations in this case is \[
		|F_{n-2}(321,31524)|=Q_{n-2}.
		\]
		\item $\pi=k\ 1\cdots$ where $3\leq k\leq n$. The forbidden pattern 321 implies that $2,\ldots,k-1$ are in increasing order. Let\[
		\pi^{\prime}=k\ 1\ 2\cdots k-1.
		\]
		We should insert $k+1,\ldots,n$ into $\pi^{\prime}$ to get $\pi$.
		 Next, we classify the permutations based on whether there are elements to the left of $k-1$   in the range  $k+1$   through  $n$.

		\begin{enumerate}[a)]
			\item There are no elements in $k+1,\ldots,n$ to the left of $k-1$. Then we can consider $k\ 1\ 2\cdots k-1$ as the smallest element. The number of permutations is\[
			|F_{n-k}(321,31524)|=Q_{n-k}.
			\]
			\item There are elements to the left of $k-1$. Assume that $k+1\leq \ell\leq n$ is the largest among them. We claim that the elements $k+2,\ldots,\ell-1$ are to the left of $\ell$ in increasing order. If there exists an element $x$ smaller than $\ell$ that is to the right of $k-1$, then $k,1,\ell,k-1,x$ form a copy of 31524.

			Thus, we can insert $k+1,\ldots,\ell$ into the $k-2$ sites between 1 and $k-1$ in increasing order. The enumeration can be transformed into counting the number of non-negative integer solutions of \[
			x_1+x_2+\cdots+x_{k-2}=\ell-k,
			\]which is $\binom{\ell-3}{\ell-k}$.

			Next, we consider the first $\ell$ elements as a cohesive unit, which is a rearrangement of $1,2,\ldots,\ell$, and they play the role of the smallest element. By the induction hypothesis, $$|F_{n-\ell}(321,31524)|=Q_{n-\ell}.$$

			In this subcase, the number of permutations is\[
			\sum_{\ell=k+1}^{n}\binom{\ell-3}{\ell-k}Q_{n-\ell}.
			\]
		\end{enumerate}
	\end{enumerate}
Therefore, we have $$
	|F_n(321,31524)|=Q_{n-1}+Q_{n-2}+\sum_{k=3}^{n}\Big(Q_{n-k}+\sum_{\ell=k+1}^{n}\binom{\ell-3}{\ell-k}Q_{n-\ell}\Big)=Q_{n}, \label{eq:A}
$$
where the last equality shall be proved in Appendix.
\end{proof}

\subsection{Enumerating $F_n(321,41523)$}
\begin{thm}\label{321-41523}
	For $n\geq 1$,\[
	|F_n(321,41523)|=\dfrac{P_n+P_{n-1}+1}{2}.
	\]
\end{thm}

\begin{proof}
	Our proof is by induction on $n$ at the same time, with easily checkable base cases, for
	\begin{equation}\label{41523}
		|F_n(321,41523)|=Q_{n}
	\end{equation}
	and \begin{equation}\label{41523-2}
		|F_n^{(2)}(321,41523)|=P_{n-1}.
	\end{equation}
	Suppose that the above two equality hold for positive integers less than $n$. We shall prove them hold for $n$ as well.

	\begin{enumerate}[1)]
			\item $\pi=1\oplus\tau$ with $\tau\in F_{n-1}(321,41523)$. By induction hypothesis, the number of permutations is\[
		|F_{n-1}(321,41523)|=Q_{n-1}.
		\]
		\item $\pi=2\ 1\oplus\tau$ with $\tau\in F_{n-2}(321,41523)$. The number of permutations in this case is \[
		|F_{n-2}(321,41523)|=Q_{n-2}.
		\]
		\item $\pi=3\ 1 \cdots.$ We proceed based on the position of 2.
		\begin{enumerate}[a)]
			\item $\pi=3\ 1\ 2\oplus\tau$ with $\tau\in F_{n-3}(321,41523)$. The number of permutations is \[
			|F_{n-3}(321,41523)|=Q_{n-3}.
			\]
			\item $\pi=3\ 1\ \ell\ 2\cdots$ with $4\leq\ell\leq n$. By equality \eqref{41523-2}, the number of permutations in this subcase is\[
			|F_{n-2}^{(2)}(321,41523)|=P_{n-3}.
			\]
			\item There are at least two elements between 1 and 2. We claim that the form of permutations in this case is \[
			\pi=3\ 1\ 4\cdots m\ \ell\ 2\cdots,\quad 4\leq m\leq n-1,\ m+1\leq\ell\leq n.
			\]
			First we prove $\pi_3=4$. Assuming $\pi_3\neq 4$, we consider the relative order of $\pi_3$ and $\pi_4$. To avoid the pattern 321, it must be the case that $\pi_3<\pi_4$. Consequently, $\pi_3,\pi_4,\pi_3-1$ form a copy of  $\pattern$, contradicting the fact that $\pi$ avoids this pattern.
			Furthermore, to avoid $\pattern$, we deduce that $\pi_4=5$, $\pi_5=6$ and so forth until $\pi_{m-1}=m$.
			In this particular subcase, when performing the enumeration, we can disregard the elements to the left of $\ell$, as they cannot participate in every occurrence of the patterns  $ \pattern $, 321, or 41523. Instead, we just focus  on the permutations of length $n-m+1 $ whose first element is greater than 1 and second element is 1. Thus the number of permutations in this subcase is
			\[
			|F_{n-m+1}^{(2)}(321,41523)|=P_{n-m}.
			\]
		\end{enumerate}
		So the number of permutations in this case is\[
	Q_{n-3}+P_{n-3}+\sum_{m=4}^{n-1}P_{n-m}=P_{n-2}.
	\]
	\item $\pi=k\ 1\cdots$ with $4\leq k\leq n$. To avoid 321, the elements $2,\ldots,k-1$ are in increasing order. Let\[
	\pi^{\prime}=k\ 1\ ^{1}\ 2\ ^{2}\cdots\ k-2\ ^{k-2}\ k-1\ ^{k-1}.
	\]We claim that every element $x$ with $k+1 \le x \le n$ cannot appear to the left of $k-2$. If not, $k,1,x,k-2,k-1$ form a copy of 41523. Thus we can treat $1\ 2\cdots k-2$ as a single element playing the role of 1. As a result, the enumeration reduces to counting the number of permutations of length $n-k+3$ where the first element is 3 and the second element is 1, namely the enumeration in case 3). The number of permutations is then $P_{n-k+1}$.
	\end{enumerate}
Therefore, we have \[
|F_n^{(2)}(321,41523)|=Q_{n-2}+P_{n-2}+\sum_{k=4}^{n}P_{n-k+1} =P_{n-1},
\]
and consequently
\[
|F_n(321,41523)|=|F_n^{(1)}(321,41523)|+|F_n^{(2)}(321,41523)|=Q_{n-1}+P_{n-1}=Q_{n}.
\]

\vspace{-1cm}
\end{proof}

\appendix

\setcounter{section}{1}
\section*{Appendix}\label{appendix}

In the appendix, we prove some identities about Pell numbers which are used in Section 4.

\begin{lem}\label{pn-sum}
	For $n\geq 1$, we have
	\begin{equation}
		\sum_{i=1}^{n} P_i=\dfrac{P_{n+1}+P_n-1}{2}.\label{Pn-sum}
	\end{equation}
\end{lem}
\begin{proof}
	We shall prove by induction on $n$. For $n=1$, the left-hand of \eqref{Pn-sum} equals $P_1=1$. The right-hand of \eqref{Pn-sum} equals  $\frac{P_2+P_1-1}{2}=\frac{2+1-1}{2}=1$. Suppose this equality holds for positive integers less than $n+1$. We shall prove it holds for $n+1$ as well.
	\begin{align*}
		\sum_{i=0}^{n+1} P_i=\sum_{i=0}^{n}P_i+P_{n+1}
	    =\frac{P_{n+1}+P_n-1}{2}+P_{n+1}
		=\frac{2P_{n+1}+P_n+P_{n+1}-1}{2}
		=\frac{P_{n+2}+P_{n+1}-1}{2}.
	\end{align*}
	We are done.
\end{proof}
Recall that $Q_n:= \frac{P_n+P_{n-1}+1}{2}$ for $n\geq1$ and define $Q_0=1$.
Direct computation leads to the following identity.
\begin{cor}
	For $n\geq 1$, we have
	\begin{equation}\label{Qn-sum}
		\sum_{i=0}^{n}Q_{i}=\dfrac{P_{n+1}+n+1}{2}.
	\end{equation}
\end{cor}
%
%
Another identity about $P_k$ is as follows.
\begin{lem}
	For $n\geq 1$, we have\begin{equation}\label{k.pn-k}
		\sum_{k=1}^{n}kP_{n-k}=\dfrac{P_{n+1}-n-1}{2}.
	\end{equation}
\end{lem}

\begin{proof}
	We shall prove the statement by induction on $n$. For $n=1$, it is easy to check both sides of \eqref{k.pn-k} equal to 0. Suppose this equality holds for $n$. We shall prove it holds for $n+1$ as well.
	$$
		\sum_{k=1}^{n+1}kP_{n+1-k}=\sum_{k=1}^{n}kP_{n-k}+\sum_{i=0}^{n}P_{i}
		=\frac{P_{n+1}-n-1}{2}+\frac{P_{n+1}+P_n-1}{2}
		=\frac{P_{n+2}-n-2}{2}.
	$$
	This completes the proof.
\end{proof}

Now we can prove the equality \eqref{q+kp} which was used in the proof of Theorem \ref{321-31452}.
\begin{proof}[Proof of equality \eqref{q+kp}]
	It follows from \eqref{Pn-sum}, \eqref{Qn-sum}, and \eqref{k.pn-k} that
	\begin{align*}
		&\, Q_{n-2}+\sum_{k=3}^n\big( (k-2) P_{n-k}+Q_{n-k} \big) \\
		=&\sum_{i=0}^{n-2}Q_{i}+\sum_{k=1}^nkP_{n-k}-P_{n-1}-2P_{n-2}-2\sum_{i=0}^{n-3}P_{i}\\
		=&\frac{P_{n-1}+n-1}{2}+\frac{P_{n+1}-n-1}{2}-2\, \dfrac{P_{n-2}+P_{n-3}-1}{2}-P_{n-1}-2P_{n-2}\\
		=&\frac{P_{n+1}+P_{n-1}}{2}-3P_{n-2}-P_{n-3}-P_{n-1}	=P_{n-1}.
	\end{align*}

	\vspace{-1cm}

	\ \ \
\end{proof}
\allowdisplaybreaks

We proceed to prove \eqref{eq:A} which was used in the proof of Theorem~\ref{31524}.

\begin{proof}[Proof of equality \eqref{eq:A}]
	We first interchange the summations of $ \sum_{k=3}^{n}\sum_{\ell=k+1}^{n}\binom{\ell-3}{\ell-k}Q_{n-\ell} $  to simplify:
	$$
		\sum_{k=3}^{n}\sum_{\ell=k+1}^{n}\binom{\ell-3}{\ell-k}Q_{n-\ell}
		=\sum_{\ell=4}^{n}\sum_{k=3}^{\ell-1}\binom{\ell-3}{\ell-k}Q_{n-\ell}
		=\sum_{\ell=4}^{n}(2^{\ell-3}-1)Q_{n-\ell}
		=\sum_{\ell=0}^{n-4}2^{n-\ell-3}Q_{\ell}-\sum_{\ell=0}^{n-4}Q_{\ell}.
	$$
	It suffices to prove that \[
	\sum_{\ell=0}^{n-4}2^{n-\ell-3}Q_{\ell}=Q_n-Q_{n-1}-Q_{n-2}-Q_{n-3}.
	\]
	The base case is easy to check. Suppose this equality holds for positive integers less than $n+1$. We shall prove it holds for $n+1$ as well.
	\begin{align*}
		&\sum_{\ell=0}^{n-3}2^{n-\ell-2}Q_{\ell}=2\sum_{\ell=0}^{n-4}2^{n-\ell-3}Q_{\ell}+2Q_{n-3}
		=2(Q_n-Q_{n-1}-Q_{n-2}-Q_{n-3})+2Q_{n-3}\\
		=&Q_{n+1}-Q_{n}-Q_{n-1}-Q_{n-2}.
	\end{align*}
	This completes the proof.
\end{proof}
\end{document}